%% file: robust_test.tex
\documentclass[10pt,journal]{IEEEtran}
\usepackage{amsmath,graphicx,bm,amssymb,amsthm,enumerate,epsfig,psfrag,cases,mathtools}
\usepackage[shortlabels]{enumitem}
\usepackage[colorlinks=true,citecolor=blue,linkcolor=red]{hyperref}
\include{ilyas_commands}

\usepackage{subfig}
\usepackage{fontenc}
\usepackage{inputenc}
\usepackage[square,sort,compress,comma,numbers]{natbib}
\usepackage{epstopdf,pst-node}

\usepackage{mathtools}

\DeclarePairedDelimiter{\floor}{\lfloor}{\rfloor}

\begin{document}
\title{A Robust Test for Elliptical Symmetry}
\author{
    \IEEEauthorblockN{Ilya Soloveychik} \\
    \IEEEauthorblockA{\normalsize The Hebrew University of Jerusalem}
}

\maketitle

\begin{abstract}
Most signal processing and statistical applications heavily rely on specific data distribution models. The Gaussian distributions, although being the most common choice, are inadequate in most real world scenarios as they fail to account for data coming from heavy-tailed populations or contaminated by outliers. Such problems call for the use of Robust Statistics. The robust models and estimators are usually based on elliptical populations, making the latter ubiquitous in all methods of robust statistics. To determine whether such tools are applicable in any specific case, goodness-of-fit (GoF) tests are used to verify the ellipticity hypothesis. Ellipticity GoF tests are usually hard to analyze and often their statistical power is not particularly strong. In this work, assuming the true covariance matrix is unknown we design and rigorously analyze a robust GoF test consistent against all alternatives to ellipticity on the unit sphere. The proposed test is based on Tyler's estimator and is formulated in terms of easily computable statistics of the data. For its rigorous analysis, we develop a novel framework based on the exchangeable random variables calculus introduced by de Finetti. Our findings are supported by numerical simulations comparing them to other popular GoF tests and demonstrating the significantly higher statistical power of the suggested technique.


\end{abstract}


\section{Introduction}
The majority of methods and techniques used by statistical signal processing and data science heavily rely on various assumptions on the data such as independence of the samples, certain parametric families of possible distributions, etc. Very rarely these assumptions are confirmed on the observed samples and even if such verification attempt is made the data almost never agrees with the assumptions made. This leads to poor inference, or even to situations in which the researcher does not know the quality of the achieved results. The main reason for the lack of such tests is the technical complexity of their analysis especially when the data is far from being Gaussian. The most popular substitute for Gaussian distributions in such cases is the family of elliptical populations. Indeed, the latter is already quite flexible to reasonably approximate the heavy tails of the real-world populations or the outliers, while still allowing rigorous analytical treatment. In this paper, we focus on one of the fundamental questions accompanying any problem of signal processing or statistical inference. Namely, how reliably does the elliptical family of distributions model the data at hand. In other words, we design a novel and easy to use Goodness-of-Fit (GoF) test that efficiently and consistently validates the ellipticity assumptions. Given independently sampled data, such a test quantitatively and reliably determines whether we can assume the data to be elliptically distributed. Since an elliptical distribution is determined by its density generator function and scatter matrix\footnote{We assume populations have zero mean, which is a very natural assumption in most applications.}, estimation of the scatter (or covariance) matrix of the population becomes a prerequisite for almost any ellipticity GoF test. Below we give a detailed exposition of the covariance estimation problem in such scenarios, and provide a detailed explanation of our construction.


\subsection{Tests for Elliptical Symmetry}
\label{sec:other_tests_intro}



Numerous ellipticity GoF tests have been proposed in the statistical and signal processing literature, however, most of them lack statistical power as discussed in Section \ref{sec:comparison} below in detail. Such tests are also rarely supported by provable analysis, since such analysis often becomes infeasible when the Gaussian assumptions are lifted. 
In addition, the computational schemes of some of these testing procedures are so complex that their usage becomes computationally infeasible even on modern machines.

In \cite{beran1979testing} Beran introduces a test based on marginal signs and ranks. That test is neither distribution-free within the family of elliptical populations nor affine-invariant. In addition, the authors do not provide practical guidelines to the choice of the basis functions involved in the test statistic making its application difficult. Baringhaus \cite{baringhaus1991testing} proposes a test for spherical symmetry of Cramer-von Mises type based on the independence between the norm of the samples and their directions. The asymptotic distribution of this test is very hard to achieve and exploit. Dyckerhoff et al. \cite{dyckerhoff2015depth} demonstrated empirically that this test can be used as a test for elliptical symmetry in dimension $2$. Koltchinskii and Sakhanenko \cite{koltchinskii2000testing} design tests using the bootstrap methodology. These tests are based on a class of functions closed under orthogonal transformations and have no known asymptotic distribution, thus requiring bootstrap to get the critical values. Manzotti et al. \cite{manzotti2002statistic} develop a test based on spherical harmonics to test whether the normalized vectors are uniformly distributed on the unit sphere. The test is computationally demanding and requires moments of order 4. Schott \cite{schott2002testing} builds a Wald-type test to compare the empirical  fourth-order moments with the expected ones under elliptical symmetry. The test is relatively simple in usage since it is based on the moments of low order. However, it has very low power against several alternatives. The authors of \cite{huffer2007test} propose a Pearson $\chi^2$-type test with multidimensional cells. Its asymptotic distribution exists only for Gaussian scenario, otherwise bootstrap techniques are required. Cassart et al. \cite{cassart2008optimal} construct a pseudo-Gaussian test that is most efficient against a multivariate form of Fechner-type asymmetry. The test requires finite moments of order $4$. Tests based on Monte Carlo simulations can be found in \cite{diks1999test} and \cite{zhu2000nonparametric}. The authors of \cite{li1997some} exploit graphical methods while \cite{zhu2003conditional} build conditional tests. We refer the reader to \cite{serfling2006multivariate} and \cite{sakhanenko2008testing} for extensive surveys and performance analysis of the aforementioned tests. 

\subsection{Covariance Estimation}
Elliptical GoF tests are almost never possible without explicit (e.g. in the plug-in form or through whitening the data) or implicit (e.g. spherical/isotropic GoF tests) estimation or stipulation of the covariance structure. In all of these tests, whenever the scatter matrix is unknown it must be estimated from the available data. In this section we provide a brief survey of the non-Gaussian covariance estimation literature and focus attention on Tyler's estimator which is later used in our construction.

Covariance estimation is a fundamental problem of its own in multivariate statistical analysis. It arises in diverse applications such as signal processing \cite{kelly1986adaptive, soloveychik2014groups}, geometric functional analysis and computational geometry \cite{adamczak2010quantitative}, genomics \cite{schafer2005shrinkage, xie2003covariance, hero2012hub}, functional MRI \cite{derado2010modeling}, modern social networks analysis \cite{lauritzen1996graphical, banerjee2008model}, empirical finance \cite{bai2011estimating, ledoit2003improved}, classical problems of clustering and discriminant analysis  \cite{friedman1989regularized}, and many other fields. Application of structured covariance matrices instead of Bayesian classifiers based on Gaussian mixture densities or kernel densities proved to be very efficient for many pattern recognition tasks, among them speech recognition, machine translation and object recognition \cite{dahmen2000structured}.


As mentioned earlier, in most real world applications the Gaussian models become unacceptable and robust covariance estimation methods that allow the populations to be heavy-tailed or contain a small proportion of outliers are required \cite{huber1964robust, maronna1976robust}. 
In the 70-s through the analysis of elliptical populations and their Maximum Likelihood (ML) estimates, R. A. Maronna discovered a family of covariance $M$-estimators \cite{maronna1976robust}. These estimators turned out to be much more robust to outliers than the classical sample covariance which is the ML estimator in the Gaussian setup. The ideas of Maronna were further developed by D. E. Tyler who derived a distribution-free robust covariance matrix estimator \cite{tyler1987distribution}. This estimator fits any population from the Generalized Elliptical (GE) family \cite{frahm2004generalized} and is also a member of the $M$-estimators family. Tyler's estimator has become very widely used by engineers \cite{abramovich2007time, pascal2008covariance, bandiera2010knowledge, wiesel2015structured} since its discovery. Although, generally $M$-estimators are given as solutions to optimization programs, Tyler showed that his estimator can be obtained as a solution to a simple fixed point equation
\begin{equation}
\label{eq:tyler_def}
\T = \frac{p}{n}\sum_{i=1}^n \frac{\x_i\x_i^\top}{\x_i^\top\T^{-1}\x_i},
\end{equation}
where $\x_1,\dots,\x_n \in \mathbb{R}^{p}$ are the collected sample vectors. To avoid the obvious scaling ambiguity (for a solution $\T$ to (\ref{eq:tyler_def}), $c\cdot\T$ is also a solution whenever $c>0$), it is common to fix the scaling, e.g. by setting $\Tr{\T} = p$. When $\{\x_i\}_{i=1}^n$ are i.i.d. (independent and identically distributed) elliptical \cite{frahm2004generalized}, their true scatter matrix $\bm\Omega$ is positive definite and $n>p$, Tyler's estimator exists with probability one and is a consistent estimator of $\bm\Omega$. In \cite{tyler1987statistical} Tyler also demonstrated that his estimator can be viewed as an ML estimator of a certain distribution over a unit sphere. In elliptical populations the scatter matrix is equal to a positive multiple of the covariance matrix when the latter exists.

The elliptical and generalized elliptical classes of distributions are quite large to incorporate many known populations and they model the real non-Gaussian world behavior much better  \cite{pascal2008covariance, abramovich2007time, wiesel2012geodesic, zhang2013multivariate} than the Gaussian distributions. In particular, the GE family includes generalized Gaussian, compound Gaussian and many other widely used distributions \cite{frahm2004generalized}. Elliptical populations are commonly used to model radar clutter \cite{conte1991modelling}, noise and interference in indoor and outdoor mobile communication channels \cite{middleton1973man} and numerous other applications.

Other robust covariance estimation approaches were also proposed, however, they have not become so much popular as $M$-estimators. Among them is the Stahel-Donoho estimate \cite{stahel1981breakdown, donoho1981breakdown}, whose main idea is to detect and down-weight outliers based on their one dimensional  ``\textit{outlyingness measure}''. Another method proposed by P. J. Rousseeuw \cite{rousseeuw1985multivariate} is the Minimum Volume Ellipsoid estimate, whose name stems from the fact that among all ellipsoids containing at least half of the data points, the one defined by the Minimum Volume Ellipsoid estimate has the minimal volume. A more efficient approach, the so-called $S$-estimator, was later proposed by P. L. Davies \cite{davies1987asymptotic} and deeply investigated by H. P. Lopuhaa \cite{lopuhaa1989relation} and O. H{\"o}ssjer \cite{hossjer1992optimality}. The Minimum Covariance Determinant estimate \cite{rousseeuw1999fast} is another possibility of robust covariance estimation. When prior knowledge on the estimator is available to the research, a Bayesian covariance estimator \cite{leonard1992bayesian, alvarez2014bayesian, yang1994estimation} would become natural. Shrinkage estimators in the paradigm of James-Stein estimator are a particular case of the Bayesian methodology. Shrinkage estimators of covariance matrices are computed as a conical (often convex) combination of a certain data statistic (e.g. sample covariance) and a constant matrix (e.g. the identity matrix) representing the prior \cite{ledoit2004honey, schafer2005shrinkage}. Robust analogs of numerous shrinkage estimators were also recently developed and thoroughly studied \cite{chen2011robust, couillet2014large}.

The behavior of Tyler's estimator had been methodically investigated in various asymptotic regimes and multiple high-probability performance bounds have been developed for its analysis \cite{soloveychik2014non, besson2013fisher, greco2013cramer, duembgen1997asymptotic, frahm2012semicircle, zhang2014marchenko, couillet2012robust, couillet2014robust, couillet2014large, pascal2007covariance, bandiera2010knowledge}. However, all of these results only hold if the samples are elliptically distributed, which is easily achievable in simulation studies but can hardly be guaranteed in real applications. Therefore a much more practical question can be formulated as follows: Given the data, verify that the ellipticity assumptions can be applied to it and therefore the tools of robust statistics will yield meaningful and reliable results when applied to it. This is the question we address in our work focusing on Tyler's estimator.

\subsection{Our Approach and Contribution}
Many of the GoF tests method mentioned in Section \ref{sec:other_tests_intro} use plug-in estimates of the covariance matrix to whiten the samples before applying GoF test of uniformity over the unit sphere. Usually such plug-in estimates are based on the sample covariance matrix (e.g. \cite{manzotti2002statistic, koltchinskii2000testing, schott2002testing, huffer2007test} and numerous others) which creates an unbalanced situation. Indeed, such tests assume ellipticity of the population, thus allowing it to have heavy-tails, but the tools used to estimate the covariance are not robust and therefore no performance guarantees in finite samples can be offered by these tests. In other cases, robust estimates are used, however, without rigorous studies and claims because of the significant level of complexity of the required analysis. When testing the ellipticity hypothesis, it is common to separate the question of elliptical symmetry from the radial density \cite{tyler1987statistical, frahm2004generalized}. Such a standard approach enables one to project the samples on the unit sphere as described in Sections \ref{sec:setup} and \ref{sec:prob_state} in detail and study the distribution of their normalized values. Remarkably, modulo positive scaling, such a transformation does not affect the scatter matrix of an elliptical (or GE) population and it is the use of this very technique that led to the discovery of Tyler's estimator as an ML estimate of the scatter matrix \cite{tyler1987statistical}. Following these ideas - for the scenario of unknown scatter matrix - we exploit Tyler's estimator to develop an asymptotically consistent GoF test against all alternatives to ellipticity on the unit $p$-dimensional sphere.\footnote{The unknown distribution is assumed to be a Lebesgue measurable probability measure over the $p$-dimensional Euclidean sphere as elaborated further. By \emph{all alternatives}, it is common to understand the set of all probability measures except for the class under consideration, which in our case is the class of elliptical (or generalized elliptical) populations \cite{huffer2007test, sakhanenko2008testing, koltchinskii2000testing, gine1975invariant}.} To enable analytical treatment of such a hypothesis test, we reformulate it as an asymptotic uniformity test for a certain stochastically dependent sequence of unit random vectors. The main tool used in the construction and analysis of the uniformity tests for i.i.d. scenario is the Central Limit Theorem (CLT) \cite{gine1975invariant, prentice1978invariant, garcia2018overview} which is clearly not applicable when the measurements are not independent. For our setup, we develop a novel toolbox that allows verification of the null hypothesis by resorting to the concept of \textit{exchangeability}. 

A sequence of variables is called \textit{exchangeable} if the joint distribution of any finite subset of these variables is invariant under arbitrary permutations of their indices. Exchangeable random variables were first introduced by de Finetti \cite{de1929funzione, de1937prevision} as a direct and natural generalization of i.i.d. sequences. Interestingly, exchangeable random variables serve as one of the fundamental building blocks of the Bayesian statistics \cite{cifarelli1996finetti}. Unlike the i.i.d. case, the behavior of exchangeable sequences is much harder to analyze. We exploit certain versions of CLT and the Strong Law of Large Numbers (SLLN) for exchangeable variables to demonstrate asymptotic consistency of our test statistics built analogously to the generalized Ajne and Gin{\'e} statistics \cite{ajne1968simple, gine1975invariant, prentice1978invariant} developed for the i.i.d. case.\footnote{The Ajne statistic was originally introduced for distributions on a circle \cite{ajne1968simple}, the idea was extended by \cite{beran1968testing} to the $2$-dimensional unit sphere and later generalized by \cite{prentice1978invariant} for the $p-1$-dimensional spheres. Similarly, Gin{\'e}'s statistic was originally defined for $1$- and $2$-dimensional spheres and later extended by \cite{prentice1978invariant} to the general dimension.}

Following Tyler, our approach becomes essentially \emph{distribution-free} within the elliptical family since we do not focus on estimating the radial density function \cite{babic2019optimal}. We offer a test which is consistent against \emph{all} alternatives to elliptical symmetry and not only certain classes of densities \cite{babic2019optimal, manzotti2002statistic, cassart2008optimal}. We do not use the sample covariance matrix as a plug-in estimator in particular because its convergence to the true covariance in elliptical populations maybe very slow due to heavy tails \cite{manzotti2002statistic}, however, we emphasize that our technique \textit{allows} the use of the sample covariance instead of Tyler's estimator. Our test is not limited to certain moments which makes it more natural and less computationally demanding \cite{schott2002testing, cassart2008optimal}. Unlike Monte Carlo simulations-based techniques \cite{diks1999test, zhu2000nonparametric}, our methodology is rigorous and offers a deterministic and computationally cheap algorithm. Finally, unlike previously mentioned works (and articles referenced in \cite{babic2019optimal}) we believe that the methodology based on the \textit{exchangeability} framework is the most suitable for the analysis of populations transformed by plug-in estimators. This type of analysis is usually technically more complex but in our eyes it represents the natural approach to the problem. Our theoretical studies are supported by extensive numerical simulations featuring the properties of the tests and comparing them to other available tests mentioned in earlier in Section \ref{sec:other_tests_intro}.

The rest of the article is organized as follows. In Section \ref{sec:setup}, we introduce the setup and notation. The problem is formulated in Section \ref{sec:prob_state} where we also present some of the existing tests for the known scatter case. In Section \ref{sec:reform_aux}, we reformulate the problem and introduce necessary background on exchangeable random variables; Section \ref{sec:distr_null} provides some additional notation and auxiliary results. In Section \ref{sec:asymp_unif_exch}, we formulate the main results and discuss them. Section \ref{sec:num_sims} provides numerical studies of the proposed tests while Section \ref{sec:comparison} demonstrates their power in comparison to other commonly used ellipticity tests. The conclusion is provided in Section \ref{sec:concl}. Some of the proofs are postponed to the Appendix.

\section{Notation and Setup}
\label{sec:setup}
\begin{definition}[\cite{frahm2007tyler}]
\label{def:ellipt_def}
A vector $\y \in \mathbb{R}^p$ is elliptically distributed with the scatter matrix $\bm\Omega \succ 0$ and mean $\bm\mu$ if there exists a random vector $\w \in \mathcal{S}^{p-1}$ uniformly distributed over the unit $p-1$-dimensional sphere and an independent random variable $r\geqslant 0$, such that
\begin{equation}
\label{eq:def_cent_ellip}
\y = \bm\mu + r\cdot\bm\Omega^{1/2}\w.
\end{equation}
\end{definition}

For example, if $r \sim \sqrt{\chi_p^2}$, then $\y \sim \mathcal{N}\(\bm\mu, \bm\Omega\)$. In what follows we always assume that the data is centered, $\bm\mu=0$. Let us consider the normalized vector,
\begin{equation}
\label{eq:x_norm_def}
\x = \frac{\y}{\norm{\y}} = \frac{\bm\Omega^{1/2}\w}{\norm{\bm\Omega^{1/2}\w}},
\end{equation}
which can be equivalently viewed as disregarding the information stored in the scalar variable $r$ but keeping the information provided by the scatter matrix. As we see below, the distribution of $\x$ contains all the information about the scatter matrix $\bm\Omega$. We are going to recover the scatter matrix by sampling from the distribution of $\x$. Denote by $\I=\I_p$ the $p$-dimensional identity matrix.

\begin{definition}[\cite{tyler1987statistical}]
\label{def:ellip_angular}
The family of real Angular Central Gaussian (ACG) distributions on $\mathcal{S}^{p-1}$ is defined by the densities of the form
\begin{equation}
p(\x;\bm\Omega) = \frac{\Gamma(p/2)}{2\pi^{p/2}\det{\bm\Omega}^{1/2}}\frac{1}{(\x^\top\bm\Omega^{-1}\x)^{p/2}},\quad \x \in \mathcal{S}^{p-1},
\end{equation}
for $\bm\Omega \succ 0$ which is called the scatter matrix.
\end{definition}
When $\x$ is ACG distributed with the scatter matrix $\bm\Omega$, we write
\begin{equation}
\x \sim \mathcal{U}\(\bm\Omega\),
\end{equation}
in particular when $\bm\Omega=\I$ we get the uniform distribution over the unit sphere $\mathcal{U}\(\I\)$. Note that ACG is not a member of the elliptical family but actually belongs to a wider class of \textit{generalized} elliptical populations whose definition is identical to Definition \ref{def:ellipt_def} except for weakened assumptions on $r$ \cite{frahm2007tyler}. In generalized elliptical population, $r$ does not have to be stochastically independent of $\w$ and does not have to be non-negative. The following result allows us to reduce estimation of the scatter matrices of elliptical populations to the estimation of the scatter matrices of ACG vectors.
\begin{lemma}[\cite{frahm2007tyler}]
For a random vector $\y$ sampled from a centered elliptical population with the scatter matrix $\bm\Omega$, $\x$ defined in (\ref{eq:x_norm_def}) is ACG distributed with the same scatter matrix.
\end{lemma}

Now assume $n > p$ i.i.d. random vectors $\x_1,\dots,\x_n \in \mathbb{S}^{p-1}$ are sampled from $\mathcal{U}\(\bm\Omega\)$, then as shown in \cite{tyler1987statistical} the ML estimator of the scatter matrix exists almost surely and is given by the fixed point equation (\ref{eq:tyler_def}). The solutions to this equation form a ray since the latter is invariant under multiplication of the matrix $\T$ by a positive constant. To resolve the ambiguity we choose $\T$ to satisfy $\Tr{\T} = p$, however, we note that the specific choice of the scaling does not affect any of the results presented below.

\section{Problem Formulation and State of the Art}
\label{sec:prob_state}
\subsection{Main Goal}
The problem considered in this article can be formulated as follows. Given a sequence of vectors $\{\x_i\}_{i=1}^n \subset \mathbb{S}^{p-1}$ sampled independently, we want to test two alternative hypotheses,
\begin{align}
\mathcal{H}_0 & : \x_1,\dots,\x_n \;\overset{\text{i.i.d.}}{\scalebox{1.5}[1]{$\sim$}}\;\; \mathcal{U}(\bm\Omega), \text{ for some } \bm\Omega, \label{eq:main_hyp_test_1}\\
\mathcal{H}_1 & : \x_1,\dots,\x_n \;\overset{\text{i.i.d.}}{\scalebox{1.5}[1]{$\nsim$}}\;\; \mathcal{U}(\bm\Omega), \text{ for any } \bm\Omega, \label{eq:main_hyp_test_2}
\end{align}
and in the case of $\mathcal{H}_0$ we want to estimate the scatter matrix $\bm\Omega$, as well.

The test (\ref{eq:main_hyp_test_1})-(\ref{eq:main_hyp_test_2}) is a composite hypothesis since the scatter matrix is unknown. When the scatter matrix is known, the problem can be equivalently reformulated as a uniformity test on the sphere as shown below.

\subsection{Uniformity Tests on $\mathbb{S}^{p-1}$}
\label{sec:unif_test}
Assume the scatter matrix $\bm\Omega$ in the hypothesis test (\ref{eq:main_hyp_test_1})-(\ref{eq:main_hyp_test_2}) is known and introduce a derived i.i.d. sequence,
\begin{equation}
\label{eq:w_i_def}
\w_i = \frac{\bm\Omega^{-1/2}\x_i}{\norm{\bm\Omega^{-1/2}\x_i}},\quad i=1,\dots,n.
\end{equation}
Under $\mathcal{H}_0$, $\w_1,\dots,\w_n \;\scalebox{1.5}[1]{$\sim$}\;\; \mathcal{U}(\I)$ and therefore the test (\ref{eq:main_hyp_test_1})-(\ref{eq:main_hyp_test_2}) becomes actually a uniformity test on the unit sphere,
\begin{align}
\mathcal{G}_0 & : \w_1,\dots,\w_n \;\overset{\text{i.i.d.}}{\scalebox{1.5}[1]{$\sim$}}\;\; \mathcal{U}(\I), \label{eq:u_0_hyp} \\
\mathcal{G}_1 & : \w_1,\dots,\w_n \;\overset{\text{i.i.d.}}{\scalebox{1.5}[1]{$\nsim$}}\;\; \mathcal{U}(\I).
\end{align}
Next, we summarize two uniformity tests on $\mathbb{S}^{p-1}$ concluding this section with Proposition \ref{prop:main_unif_test_iid} providing a uniformity test consistent against all alternatives on the unit sphere. Based on it, we will develop an analogous test for (\ref{eq:main_hyp_test_1})-(\ref{eq:main_hyp_test_2}) with unknown scatter matrix in the subsequent sections. Denote by
\begin{equation}
V_{p-1} = \int_{\x \in \mathbb{S}^{p-1}} d\x = \frac{2\pi}{\Gamma\(\frac{p}{2}\)}
\end{equation}
the area of the unit sphere. In addition, by 
\begin{equation}
\label{eq:psi_not}
\psi_{ij} = \arccos(\x_i^\top\x_j)
\end{equation}
we denote the angular separation (the shortest great circle distance) between $\x_i$ and $\x_j$ and by
\begin{equation}
N(\y) = |\{\x_i \mid \y^\top\x_i \geqslant 0\}|,\quad \y \in \mathbb{S}^{p-1},
\end{equation}
the number of points falling into the hemisphere with the pole at $\y$. Denote also
\begin{equation}
\alpha = \frac{p}{2}-1,
\end{equation}
\begin{equation}
\nu(a,b) = {a+b-2 \choose a-1} + {a+b-1 \choose a-1}.
\end{equation}

The following two popular statistics and detailed investigation of their behavior can be found in \cite{ajne1968simple, gine1975invariant}. These results were later generalized in \cite{prentice1978invariant} and summarized in \cite{garcia2018overview}. 

\begin{prop}[Generalized Ajne Test, \cite{ajne1968simple, prentice1978invariant}]
Under the uniformity hypothesis, the Ajne statistic
\begin{equation*}
t_{A} = \frac{1}{nV_{p-1}} \int_{\y \in \mathbb{S}^{p-1}}\(N(\y)-\frac{n}{2}\)^2 d\y = \frac{n}{4}-\frac{1}{\pi n}\sum_{i<j}\psi_{ij}
\end{equation*}
is asymptotically distributed as $\mathcal{L}\(\sum_{q=1}^\infty a_{2q-1}^2K_{\nu(p-1,2q-1)}\)$, where $K_{\xi}$ are independent random variables distributed as $\chi_{\xi}^2$ and
\begin{equation}
a_{2q-1} = \frac{(-1)^{q-1}2^{p-2}\Gamma(\alpha+1)\Gamma(q+\alpha)(2q-2)}{\pi(q-1)!(2q+p-3)!}.
\end{equation}
\end{prop}

\begin{prop}[Generalized Gin{\'e} Test, \cite{gine1975invariant, prentice1978invariant}]
\label{prop:gine_origin}
Under the uniformity hypothesis, the Gin{\'e} statistic
\begin{equation}
t_{G} = \frac{n}{2}-\frac{p-1}{2n}\(\frac{\Gamma\(\alpha+\frac{1}{2}\)}{\Gamma(\alpha+1)}\)^2\sum_{i<j}\sin(\psi_{ij})
\end{equation}
is asymptotically distributed as $\mathcal{L}\(\sum_{q=1}^\infty a_{2q}^2K_{\nu(p-1,2q)}\)$, where $K_{\xi}$ are independent random variables distributed as $\chi_{\xi}^2$ and
\begin{equation}
a_{2q}^2 = \frac{(p-1)(2q-1)}{8\pi(2q+p-1)}\(\frac{\Gamma\(\alpha+\frac{1}{2}\)\Gamma\(q-\frac{1}{2}\)}{\Gamma\(q+\alpha+\frac{1}{2}\)}\)^2.
\end{equation}
\end{prop}

The following statement provides a concise and directly applicable test for uniformity under the assumption that the random vectors are sampled i.i.d. from $\mathcal{U}(\I)$.

\begin{prop}[Uniformity test, \cite{gine1975invariant, prentice1978invariant}]
\label{prop:main_unif_test_iid}
Any weighted sum of $t_A$ and $t_G$ is consistent against all alternatives to uniformity on $\mathbb{S}^{p-1}$.
\end{prop}

In practice, one way to make the decision about accepting or rejecting $\mathcal{H}_0$ or $\mathcal{G}_0$ is as follows. The statistician truncates the series mentioned in the last two propositions in a data-driven manner and compares the sample values of $t_A$ and $t_G$ with the tables (or explicit numerical approximations) of the corresponding distributions. Another more general approach consists in replacing $t_A$ and $t_G$ by statistics whose expansions only have finite number of non-zero coefficients $a_k$ (see \cite{gine1975invariant} for more details). An efficient data-driven approach to the design of the uniformity tests based on a modification of the Bayesian Information Criterion was developed by \cite{jupp2008data}. Often in practice, the distributions of the statistics at hand under the null hypothesis are estimated empirically, by generating samples from the latter. This is the approach adopted in Section \ref{sec:comparison} below.

In this paper we are interested in the case of unknown scatter matrix in (\ref{eq:main_hyp_test_1})-(\ref{eq:main_hyp_test_2}). As we see below this makes the hypothesis test much more involved. In the next sections we develop analogs of generalized Ajne and Gin{\'e} uniformity tests for this scenario.

\section{Problem Reformulation and Exchangeability}
\label{sec:reform_aux}
\subsection{Methodology}
From Theorem 3.1 from \cite{tyler1987distribution} we know that under $\mathcal{H}_0$ Tyler's estimator converges almost surely to the true scatter matrix when $n\to\infty$. This idea motivated our study of a new sequence of vectors, defined as follows. Under $\mathcal{H}_0$ introduced in (\ref{eq:main_hyp_test_1}), we now consider the sequence
\begin{equation}
\t_i = \frac{\T^{-1/2}\x_i}{\norm{\T^{-1/2}\x_i}} \in \mathbb{S}^{p-1},\quad i=1,\dots,n,
\end{equation}
where $\T$ is defined in (\ref{eq:tyler_def}). The main challenge we face in the study of $\{\t_i\}$ is the lack of stochastic independence unlike the case of $\{\w_i\}$ defined in (\ref{eq:w_i_def}). Indeed, most existing convergence results explicitly rely on independence in their derivations in such a way that any deviation from this assumption ruins the performance analysis. For example, all the results of Ajne, Gin{\'e}, and Prentice utilize the CLT and thus require independence as the most crucial assumption \cite{ajne1968simple, gine1975invariant, prentice1978invariant, garcia2018overview}.  

Next we include a brief summary of the exchangeability concept and the related toolbox. We then use it in Section \ref{sec:distr_null} to overcome the loss of independence in our analysis of the consistency of $\{\t_i\}$ and their statistics.

\subsection{Exchangeable Random Variables}
\begin{definition}
Given a sequence $\{X_i\}$ (finite or infinite) of random variables, we say that it is exchangeable if the joint distribution of any finite subset of variables is invariant under arbitrary permutations of their indices.
\end{definition}
In other words, exchangeability is our indifference to the order of the measurements. This is clearly a much weaker hypothesis than independence, as any i.i.d. sequence is obviously exchangeable. In his seminal works de Finetti \cite{de1929funzione, de1937prevision} demonstrated that in certain sense every (infinite) exchangeable sequence can be represented as a composition of sequences of i.i.d. variables. This result can be viewed as the analog of Fourier decomposition in analysis, as it allows one to represent a \textit{more complicated} exchangeable sequence as a superposition of basic building blocks - independent sequences - objects much easier accessible for analysis and reasoning.

De Finetti \cite{de1929funzione, de1937prevision} and some of his followers focused on infinite exchangeable sequences. There exist, however, finite sets of exchangeable random variables which cannot be embedded into infinite sequences, these are called \textit{finitely exchangeable} or \textit{non-extendable}. The analysis of extendable sequences can be reduced to the analysis of infinite sequences. On the other hand, the non-extendable sequences require quite different approaches \cite{konstantopoulos2019extendibility}. Our sequence of samples $\{\t_i\}_{i=1}^n$ is an example of a non-extendable exchangeable sequence of random vectors. Indeed, their order obviously does not matter since $\T$ is not affected by permutations of the measurements $\{\x_i\}_{i=1}^n$. We can also see that this sequence is non-extendable, since addition of new random vectors $\x_j$ without an amendment of $\T$ will turn the sequence into non-exchangeable. For a detailed study of non-extendability we refer the reader to \cite{konstantopoulos2019extendibility} and references therein.

The main result of our paper can be briefly summarized as follows. We demonstrate that the limiting behavior of the samples $\{\t_i\}_{i=1}^n$ is in certain sense analogous to the behavior of the vectors uniformly distributed over the unit sphere and therefore, we can apply similar tools for the hypothesis tests. Below we show how to overcome the technical challenges on this way.

\subsection{Limit Theorems for Exchangeable Variables}
To illustrate the previous section and better describe the nature of the exchangeability phenomenon and its relation to the stochastic independence, in this section we present analogs of the SLLN and CLT for triangular arrays of exchangeable variables.
\begin{lemma}[Strong Law of Large Numbers for Exchangeable Arrays]
\label{lem:SLLN}
Let $\{X_{ni}\}_{n,i=1}^{\infty,n}$ be a triangular array of row-wise exchangeable random variables and $\{X_{\infty i}\}_{i=1}^{\infty}$ be a sequence of exchangeable random variables of bounded second moment such that
\begin{enumerate}
\item $X_{n1} \xrightarrow{a.s.} X_{\infty 1},\;\; n\to\infty$,
\item $\var{X_{n1}-X_{\infty 1}} \to 0,\;\; n\to\infty$,
\item $\mathbb{E}\[X_{n1}X_{n2}\] \to 0,\;\; n\to\infty$.
\end{enumerate}
Then
\begin{equation}
\frac{1}{n}\sum_{i=1}^{n}X_{ni} \xrightarrow{a.s.} 0,\;\; n \to \infty.
\end{equation}
\end{lemma}
\begin{proof}
The proof can be found in the Appendix.
\end{proof}

Let $k_n < n$ be two sequences of natural numbers such that
\begin{equation}
\label{eq:m_n_cond}
\frac{k_n}{n} \to \gamma \in [0,1).
\end{equation}

\begin{lemma}[Central Limit Theorem for Exchangeable Arrays, Theorem 2 from \cite{weber1980martingale}\footnote{To simplify the notation we assume the number of the elements in the $n$-th row to be $n$ unlike the seemingly more general case of $m_n$ variables considered in \cite{weber1980martingale}.}]
\label{eq:martin_theorem}
Let $\{X_{ni}\}_{n,i=1}^{\infty,n}$ be a triangular array of row-wise exchangeable random variables such that
\begin{enumerate}
\item $\mathbb{E}\[X_{n1}X_{n2}\] \to 0,\;\; n \to \infty$,
\item $\max\limits_{1\leqslant i \leqslant n} \frac{|X_{ni}|}{\sqrt{n}} \xrightarrow{P} 0,\;\; \forall n$,
\item $\frac{1}{k_n}\sum_{i=1}^{k_n} X_{ni}^2 \xrightarrow{P} 1,\;\; n \to \infty$.
\end{enumerate}
Then
\begin{equation*}
\sqrt{k_n}\[\frac{1}{k_n}\sum_{i=1}^{k_n}X_{ni} - \frac{1}{n}\sum_{i=1}^{n}X_{ni}\] \xrightarrow{L} \mathcal{N}(0,1-\gamma),\;\; n \to \infty.
\end{equation*}
\end{lemma}

As mentioned earlier this result provides an analog of the CLT for exchangeable sequences. However, it is important to stress its distinction from the classical CLT-type claims for the i.i.d. variables. Indeed, Lemma \ref{eq:martin_theorem} only allows us to consider a subset of the sample of cardinality $k_n$ smaller than the number of variables $n$ in the row so that even their ratio must not approach one. This is a reflection of the essential difference between non-extendable exchangeable sequences and their extendable counterparts that include i.i.d. sequences as a particular case \cite{konstantopoulos2019extendibility}.

\subsection{Additional Notation and Auxiliary Results}
\label{sec:distr_null}

Assume that an infinite i.i.d. sequence $\{\x_i\}_{i=1}^\infty$ is sampled under the composite $\mathcal{H}_0$ with the true scatter matrix is unknown. For every $n > p$, let the sequence of corresponding Tyler's estimators be
\begin{equation}
\T_n = \frac{p}{n}\sum_{i=1}^n \frac{\x_i\x_i^\top}{\x_i^\top\T_n^{-1}\x_i},\;\; n=p+1,\dots,
\end{equation}
which exist almost surely for a random sample \cite{wiesel2012geodesic, wiesel2012unified}. Consider a triangular array of row-wise exchangeable random vectors
\begin{equation}
\t_{ni} = \frac{\T_n^{-1/2}\x_i}{\norm{\T_n^{-1/2}\x_i}} \in \mathbb{S}^{p-1},\quad i=1,\dots,n,\;\; n=p+1,\dots.
\end{equation}
Note that by Definition \ref{def:ellipt_def}, the sequence $\{\x_i\}_{i=1}^\infty$ can equivalently be defined as follows. Given a sequence $\{\w_i\}_{i=1}^\infty \sim \mathcal{U}(\I_p)$ of uniform i.i.d. random vectors, we look at their transforms
\begin{equation}
\label{eq:constr_x}
\x_i = \frac{\bm\Omega^{1/2}\w_i}{\norm{\bm\Omega^{1/2}\w_i}},
\end{equation}
for some fixed but unknown $\bm\Omega \succ 0$. Define also an auxiliary sequence
\begin{equation}
\t_{\infty i} = \w_i.
\end{equation}
\begin{lemma}
\label{lem:slln_mean}
With the notation introduced above,
\begin{equation}
\t_{ni} \xrightarrow{a.s.} \t_{\infty i},\;\; n\to\infty.
\end{equation}
\end{lemma}
\begin{proof}
The proof can be found in the Appendix.
\end{proof}

We are now interested in the empirical distributions of the rows of the obtained triangular array, which are the finite sets $\{\t_{ni}\}_{i=1}^n$ for every fixed $n > p$. The following CLT-type result holds in our scenario.

\begin{prop}
\label{prop:main_prop}
For the triangular array of vectors $\{\t_{ni}\}_{n=p+1,i=1}^{\infty,n}$ defined above,
\begin{equation}
\sqrt{p}\cdot\frac{1}{\sqrt{n}}\sum_{i=1}^{n}\t_{ni}\; \xrightarrow{L}\; \mathcal{N}(0,\I),\;\; n\to\infty.
\end{equation}
\end{prop}
\begin{proof}
The proof can be found in the Appendix.
\end{proof}

\begin{corollary}
\label{cor:diff_clt}
Under $\mathcal{H}_0$, for any differentiable function $f\colon\mathbb{S}^{p-1}\to\mathbb{R}$,
\begin{equation}
\sqrt{p}\cdot\frac{1}{\sqrt{n}}\sum_{i=1}^{n}f(\t_{ni})\; \xrightarrow{L}\; \mathcal{N}(0, \norm{\nabla f(0)}^2),\;\; n\to\infty.
\end{equation}
\end{corollary}
\begin{proof}
The proof follows the i.i.d. case verbatim using the Maclaurin expansion of $f$.
\end{proof}

\section{Asymptotic Uniformity Tests for Exchangeable Vectors}
\label{sec:asymp_unif_exch}
In Section \ref{sec:unif_test} we introduced statistics $t_A$ and $t_G$ to test the null hypothesis of uniformity for independent samples over the unit sphere $\mathbb{S}^{p-1}$. Our next statements constitute analogs of those result for the row-wise exchangeable array $\{\t_{ni}\}_{n=p+1,i=1}^{\infty,n}$. Denote
\begin{equation}
\label{eq:psi_n_not}
\psi_{n,ij} = \arccos(\t_{ni}^\top\t_{nj}).
\end{equation}

\begin{prop}[Generalized Ajne Test for $\t_{ni}$]
\label{prop:ajne_exch}
Under $\mathcal{H}_0$, the Ajne statistic
\begin{equation}
t_{A}\(\{\t_{ni}\}\) = \frac{n}{4}-\frac{1}{\pi n}\sum_{i<j}\psi_{n,ij}
\end{equation}
is asymptotically distributed as $\mathcal{L}\(\sum_{q=1}^\infty a_{2q-1}^2K_{\nu(p-1,2q-1)}\)$ as $n\to\infty$, where $K_{\xi}$ are independent random variables distributed as $\chi_{\xi}^2$ and
\begin{equation}
a_{2q-1} = \frac{(-1)^{q-1}2^{p-2}\Gamma(\alpha+1)\Gamma(q+\alpha)(2q-2)}{\pi(q-1)!(2q+p-3)!}.
\end{equation}
\end{prop}
\begin{proof}
The proof follows \cite{gine1975invariant} and \cite{prentice1978invariant} verbatim using Corollary \ref{cor:diff_clt}.
\end{proof}

\begin{definition}
Random variable $X$ is said to be first-order stochastically dominated by random variable $Y$ if
\begin{equation}
\mathbb{P}[X \in A] \leqslant \mathbb{P}[Y \in A],
\end{equation}
for any measurable set $A$.
\end{definition}

\begin{prop}[Generalized Gin{\'e} Test for $\overline{\t}_{nk}$]
\label{prop:gine_exch}
Under $\mathcal{H}_0$, the Gin{\'e} statistic
\begin{equation}
\label{eq:gine_stat_def}
t_{G}\(\{\t_{ni}\}\) = \frac{n}{2}-\frac{p-1}{2n}\(\frac{\Gamma\(\alpha+\frac{1}{2}\)}{\Gamma(\alpha+1)}\)^2\sum_{i<j}\sin(\psi_{n,ij})
\end{equation}
is asymptotically first-order stochastically dominated by the random variable distributed as $\sum_{q=1}^\infty a_{2q}^2K_{\nu(p-1,2q)}$, where $K_{\xi}$ are independent random variables distributed as $\chi_{\xi}^2$ and
\begin{equation}
a_{2q}^2 = \frac{(p-1)(2q-1)}{8\pi(2q+p-1)}\(\frac{\Gamma\(\alpha+\frac{1}{2}\)\Gamma\(q-\frac{1}{2}\)}{\Gamma\(q+\alpha+\frac{1}{2}\)}\)^2,
\end{equation}
and
\begin{equation*}
\mathbb{E}\[t_{G}\(\{\w_i\}\)\] - \mathbb{E}\[t_{G}\(\{\t_{ni}\}\)\] \sim \frac{1}{8} + \frac{1}{16p} + O\(\frac{1}{p^2}\), \;\; n\to \infty.
\end{equation*}
\end{prop}
\begin{proof}
The proof can be found in the Appendix.
\end{proof}

It is important to emphasize the main difference between Propositions \ref{prop:gine_origin} and \ref{prop:gine_exch}. Indeed, in the former the asymptotic distribution is given by a sum of scaled $\chi^2$-variables, while in the second one the limiting distribution is first-order stochastically dominated by the same distribution and is in fact significantly \textit{thinner}. This discrepancy is due to the fact that $\{\t_{ni}\}$ are dependent in such a special way that their sample covariance matrix is exactly $\I$. The detailed reasoning and discussion can be found in the Appendix.

\begin{theorem}[Uniformity Test for $\t_{ni}$]
\label{thm:main_unif_hyp_test}
Under $\mathcal{H}_0$, any weighted sum of $t_A\(\{\t_{ni}\}\)$ and $t_G\(\{\t_{ni}\}\)$ is consistent against all alternatives to the asymptotic uniformity of $\{\t_{ni}\}$ on $\mathbb{S}^{p-1}$.
\end{theorem}
\begin{proof}
The proof follows \cite{gine1975invariant} and \cite{prentice1978invariant} verbatim using Propositions \ref{prop:ajne_exch} and \ref{prop:gine_exch}.
\end{proof}

\begin{rem}
It is also important to emphasize that an asymptotic bound to the power of the uniformity test suggested by Theorem \ref{thm:main_unif_hyp_test} against all alternatives on the unit sphere can be easily constructed using the asymptotic normality of the scaled deviations of the Ajne and Gin{\'e} statistics as shown in Section 4 of \cite{gine1975invariant}. These derivations are also valid in our case, and therefore we omit them due to a lack of space. 
\end{rem}

\section{Numerical Simulations}
\label{sec:num_sims} 
In this section, we investigate the behavior and advantages of the criterion proposed in Theorem \ref{thm:main_unif_hyp_test} through numerical simulations.
\subsection{Distributions of the Statistics under the Null Hypothesis}
In the first experiment, we compared the empirical distributions of $t_A\(\{\t_{ni}\}\)$ and $t_G\(\{\t_{ni}\}\)$ with their counterparts $t_A\(\{\w_i\}\)$ and $t_G\(\{\w_i\}\)$ for the independent samples playing the role of the benchmarks. In this simulation we took the true scatter matrix to be the identity $\bm\Omega=\I$. Figure \ref{fig:null_distr_stats} demonstrates the anticipated in Section \ref{sec:asymp_unif_exch} difference in the behavior of the Ajne and Gin{\'e} statistics. More specifically, as claimed in Proposition \ref{prop:gine_exch} and discussed in detail in its proof, the statistic $t_G\(\{\t_{ni}\}\)$ is first-order stochastically dominated by $t_G\(\{\w_i\}\)$ due to the difference in the behavior of the quadratic term in the expansion of the statistic (\ref{eq:t_g_der_sum}) caused by dependencies among $\{\t_{ni}\}$. This is in contrast to the Ajne statistic whose distributions in both cases coincide since a similar expansion into Gegenbauer polynomials involves odd degree polynomials only (for more details see the proof of Proposition \ref{prop:gine_exch}). Note also that the theoretically predicted by Theorem \ref{thm:main_unif_hyp_test} difference between the expected values for $p=8$,
\begin{equation}
\mathbb{E}\[t_{G}\(\{\w_i\}\)\] - \mathbb{E}\[t_{G}\(\{\t_{ni}\}\)\] \sim \frac{1}{8} + \frac{1}{16p} + O\(\frac{1}{p^2}\) \approx 0.133
\end{equation} 
is confirmed by the numerical simulation yielding the value of $0.131$.

\begin{figure}
\centering
\includegraphics[width=0.45\textwidth]{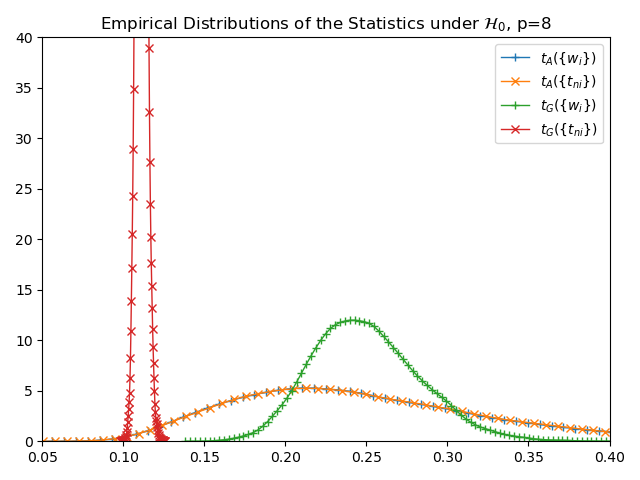}\hfill
\caption{Comparison of the empirical distributions of the Ajne and Gin{\'e} test statistics computed for the sequences $\{\w_i\}$ and $\{\t_{ni}\}$ defined in Section \ref{sec:distr_null} with the true scatter matrix being $\I$.}
\label{fig:null_distr_stats}
\end{figure}

\subsection{Criterion Performance for Alternatives}
\label{sec:altern_perf_stat}
Following Theorem \ref{thm:main_unif_hyp_test}, to demonstrate the power of the suggested methodology in our second experiment we compared the empirical distributions of the statistics
\begin{equation}
\label{eq:s_AG_def}
s\(\{\t_{ni}\}\) = t_A\(\{\t_{ni}\}\) + t_G\(\{\t_{ni}\}\),
\end{equation}
\begin{equation}
s\(\{\w_i\}\) = t_A\(\{\w_i\}\) + t_G\(\{\w_i\}\)
\end{equation}
under the null hypotheses $\mathcal{H}_0$ and $\mathcal{G}_0$ versus their distributions under specific non-elliptical alternatives $\mathcal{H}_1$ and $\mathcal{G}_1$. The alternatives were constructed as follows. We generated the uniform sequence $\{\w_i\}$ as before, added a constant offset to all the obtained vectors and re-normalized them,
\begin{equation}
\widetilde{\w}_i = \frac{\w_i + \a}{\norm{\w_i + \a}}.
\end{equation}
Note that the distributions of $\widetilde{\w}_i$ and of $\x_i$ constructed from it via (\ref{eq:constr_x})
\begin{equation}
\x_i = \frac{\bm\Omega^{1/2}\widetilde{\w}_i}{\norm{\bm\Omega^{1/2}\widetilde{\w}_i}},
\end{equation}
are not ACG and therefore our test should be able to discriminate between the hypotheses. In this experiment we chose $p=5,\, \bm\Omega = \I$ and
\begin{equation}
\a = 0.05 \cdot \frac{1}{\sqrt{p}} \begin{bmatrix} 1 \\ \vdots \\ 1\end{bmatrix}.
\end{equation}
Figure \ref{fig:conf_bands} demonstrates the $0.95$-confidence bands for the distributions of the statistics $s\(\{\t_{ni}\}\)$ and $s\(\{\w_i\}\)$ as functions of the number of measurements $n$ in the sample. The bands in Figure \ref{fig:conf_bands} were averaged over $50000$ independent trials. We see from the graph that despite the small size of $\a$, already with $n=15$ measurements the criterion allows us to easily discriminate between the hypotheses and its power is similar for both the i.i.d. (know scatter matrix) and Tyler's (unknown scatter matrix) cases.

\begin{figure*}%
    \centering
    \subfloat[$s\(\{\t_{ni}\}\)$]{{\includegraphics[width=.4\textwidth]{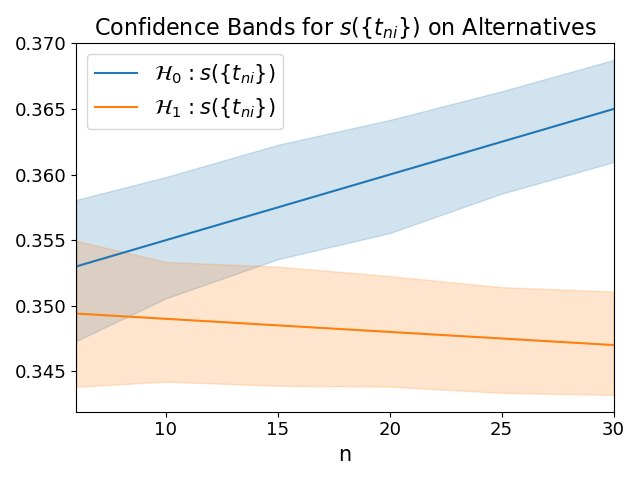} }}%
    \subfloat[$s\(\{\w_i\}\)$]{{\includegraphics[width=.4\textwidth]{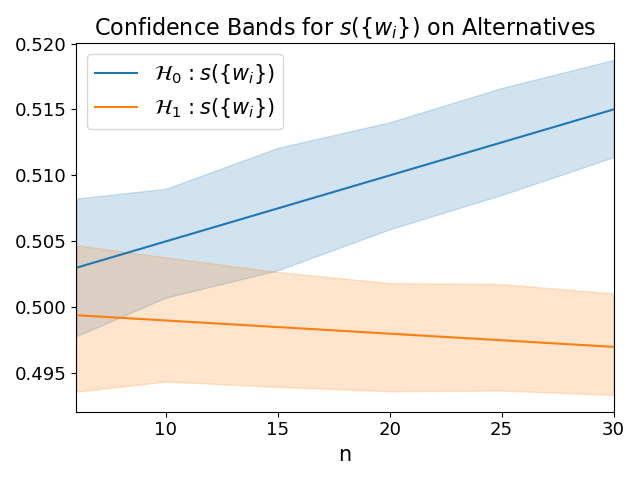} }}%
    \caption{Comparison of the $0.95$-confidence bands for $\mathcal{H}_0$ versus $\mathcal{H}_1$ designed in section \ref{sec:altern_perf_stat} for the i.i.d. $\{\w_i\}$ samples and their exchangeable $\{\t_{ni}\}$ counterparts.}
    \label{fig:conf_bands}
\end{figure*}

\section{Comparison to Other Tests}
\label{sec:comparison}
For the purpose of this comparative study, we selected the most popular tests used in the literature due to Koltchinskii and Sakhanenko \cite{koltchinskii2000testing, sakhanenko2008testing}, Manzotti et al. (MPQ test) \cite{manzotti2002statistic}, Cassart (PseudoGaussian test) \cite{cassart2008optimal}, Schott \cite{schott2002testing}, and Babic (SkewOptimal test) \cite{babic2019optimal}. Implementations of all these tests are available online trough the \textit{ellipticalsymmetry}\footnote{https://cran.r-project.org/package=ellipticalsymmetry} package developed by \cite{babic2020elliptical}.

In our numerical experiment, we set the dimension to $p=5$ and the number of samples in every batch $n=50$. Each test from the aforementioned list and the tests proposed in this paper (from Propositions \ref{prop:ajne_exch} and \ref{prop:gine_exch} and Theorem \ref{thm:main_unif_hyp_test}) is invoked to classify each such batch and determine whether it is coming from an elliptical population or not. The goal is to compare the powers of the tests using their ROC (Receiver Operating Characteristic) curves. To this end, we constructed $N=5000$ sample batches (each having $n$ vectors of dimension $p$), half of which were coming from the standard normal population and labeled as \textit{elliptical batch}, and the rest form a non-elliptical population labeled as \textit{non-elliptical batch} and defined as follows. Let
\begin{equation}
\X_{i,k} \sim \mathcal{N}(0, \I_p),\quad i=1,\dots,N,\; k=1,\dots,n,
\end{equation}
be $N$ batches of $n$ vectors each with all $\X_{i,k}$ i.i.d. standard normal. Now, we set
\begin{equation}
\Y_i = \X_i,\quad i=1,\dots,N/2
\end{equation}
for the first $N/2$ batches labeled as \textit{elliptical batch} and another $N/2$ batches marked as \textit{non-elliptical batch} we sample from a non-elliptical population designed as follows. Let
\begin{multline}
\Z_{j,i,k} \sim \mathcal{N}(\z_j, \I_p), \\ j=1,2,\; i=N/2+1,\dots,N,\; k=1,\dots,n,
\end{multline}
with the population means
\begin{equation}
\z_1 = 5\cdot\begin{bmatrix} 1 \\ 1 \\ 1 \\ 1 \\ \vdots \\ 1 \\ 1\end{bmatrix}, \quad
\z_2 = 5\cdot\begin{bmatrix} 1 \\ -1 \\ 1 \\ -1 \\ \vdots \\ 1 \\ -1\end{bmatrix}
\end{equation}
be all i.i.d. and define
\begin{equation}
\Y_i = \X_i + \Z_1 + \Z_2,\quad i=N/2+1,\dots,N.
\end{equation}
Clearly, for $N/2+1,\dots,N$  the vectors in the corresponding batches $\Y_i$ are non-elliptical and dont become GE even if normalized by their Euclidean norms. Therefore, we would expect all our tests to discriminate between \textit{elliptical} and \textit{non-elliptical} batches.

\begin{figure}
\centering
\includegraphics[width=0.45\textwidth]{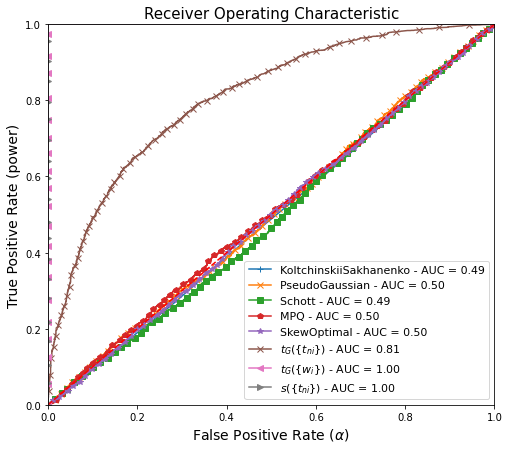}\hfill
\caption{ROC curves of Koltchinskii-Sakhanenko, MPQ, PseudoGaussian, Schott, and SkewOptimal tests in comparison to the test statistic $s\(\{\t_{ni}\}\)$ defined in (\ref{eq:s_AG_def}) and justified by Theorem \ref{thm:main_unif_hyp_test}.}
\label{fig:comparison_rocs}
\end{figure}

Figure \ref{fig:comparison_rocs} shows the ROC (Receiver Operating Characteristic) curves for all the available tests.
Naturally, the highest (having maximal power $\beta$) curve at each point on the horizontal axis corresponds to the most powerful test at this False Positive Rate ($\alpha$). We can see that the tests proposed in this article are more powerful than the rest. The distributions of the statistics $t_A\(\{\t_{ni}\}\), t_G\(\{\t_{ni}\}\), s\(\{\t_{ni}\}\)$ under $\mathcal{H}_0$ derived in Propositions \ref{prop:ajne_exch} and \ref{prop:gine_exch} were estimated empirically for this numerical study. We would like to emphasize a number of points. First, the ROC curves of all the reference tests are very close to the bisector of the first quadrant, showing their very low statistical power and making them almost equivalent to fair coin tossing. Second, every test involving $t_G\(\{\t_{ni}\}\)$ has a horizontal ROC curve and statistical power of $1$ for every $\alpha$, meaning that it correctly classified every batch. This is due to the sharp concentration of the distribution of this statistic around a constant value value as can be seen in Figure \ref{fig:null_distr_stats}. Because of that, any deviation from the expected value is a strong evidence against the ellipticity hypothesis.

\section{Conclusion}
\label{sec:concl}
In this paper we propose a novel elliptical symmetry test based on the ideas of robust statistics, and specifically Tyler's estimator of covariance matrix. This is an easy to apply and computationally cheap test with provable performance guarantees. In our extensive comparative studies we demonstrate that it surpasses all the other commonly exploited ellipticity tests by a large margin when it comes to their statistical power. In addition, based on the exchangeable random variables calculus introduced by de Finetti, we develop a natural mathematical framework enabling rigorous analysis of our test and numerous other tests and estimators based on dependent but exchangeable sample measurements. 

\appendix

\begin{proof}[Proof of Lemma \ref{lem:SLLN}]
Our proof is based on an analogous result in \cite{taylor1985strong}. Both Lemmas 1 and 2 from \cite{taylor1985strong} can be easily restated for our setup after replacing the Banach space $E$ by $\mathbb{R}$ and linear functionals by scalar multiplication. In addition, note that our condition 1) immediately implies requirement (2.5) from \cite{taylor1985strong}. Now, the reasoning from the proof of Theorem 1 from \cite{taylor1985strong} applies verbatim.
\end{proof}

\begin{proof}[Proof of Lemma \ref{lem:slln_mean}]
As shown in Theorem 3.1 from \cite{tyler1987distribution},
\begin{equation}
\T_n \xrightarrow{a.s.} \bm\Omega \succ 0,\;\; n\to\infty,
\end{equation}
therefore, starting from some $n_0,\; \T_n$ is almost surely invertible for $n\geqslant n_0$ and
\begin{equation}
\T_n^{-1/2}\bm\Omega^{1/2}\; \xrightarrow{a.s.}\; \I_p,\;\; n\to\infty.
\end{equation}
Now the claim follows from the definition of the sequence $\{\t_{ni}\}_n$,
\begin{equation*}
\t_{ni} = \frac{\T_n^{-1/2}\x_i}{\norm{\T_n^{-1/2}\x_i}} = \frac{\T_n^{-1/2}\bm\Omega^{1/2}\w_i}{\norm{\T_n^{-1/2}\bm\Omega^{1/2}\w_i}} \xrightarrow{a.s.} \w_i,\;\; n\to\infty.
\end{equation*}
\end{proof}

\begin{proof}[Proof of Proposition \ref{prop:main_prop}]
As above, we can equivalently rewrite $\t_{ni}$ as
\begin{equation*}
\t_{ni} = \frac{\T_n^{-1/2}\bm\Omega^{1/2}\w_i}{\norm{\T_n^{-1/2}\bm\Omega^{1/2}\w_i}},\quad i=1,\dots,n,\;\; n=p+1,\dots,
\end{equation*}
which is just a useful representation as clearly $\bm\Omega$ is not revealed to the researcher. Fix a vector $\a \in \mathbb{R}^p$ of unit norm $\norm{\a}=1$ and consider the following triangular array of row-wise exchangeable random variables
\begin{equation}
X_{ni} = \sqrt{p}\cdot\a^\top \t_{ni},\quad i=1,\dots,n,\;\; n=p+1,\dots.
\end{equation}
Let us study the properties of $\{X_{ni}\}_{n=p+1,i=1}^{\infty,n}$. First, consider
\begin{equation}
\mathbb{E}\[X_{n1}X_{n2}\] = p\,\a^\top \mathbb{E}\[\t_{n1}\t_{n2}^\top\]\a,
\end{equation}
Lemma \ref{lem:slln_mean} implies that
\begin{equation}
\mathbb{E}\[\t_{n1}\t_{n2}^\top\] \to \mathbb{E}\[\w_1\w_2^\top\] = 0,\quad n \to \infty,
\end{equation}
therefore,
\begin{equation}
\mathbb{E}\[X_{n1}X_{n2}\] \to 0,\quad n \to \infty.
\end{equation}
Next, note that
\begin{equation}
\frac{|X_{ni}|}{\sqrt{n}} = \sqrt{p}\frac{|\a^\top\t_{ni}|}{\sqrt{n}} \leqslant \sqrt{p}\frac{\norm{\a}\norm{\t_{ni}}}{\sqrt{n}} = \sqrt{\frac{p}{n}} \to 0.
\end{equation}
Finally, let us show that
\begin{equation}
\frac{1}{k_n}\sum_{i=1}^{k_n} X_{ni}^2 \xrightarrow{P} 1,\quad n \to \infty.
\end{equation}
Indeed,
\begin{equation*}
\frac{1}{k_n}\sum_{i=1}^{k_n} X_{ni}^2 = p\frac{1}{k_n}\sum_{i=1}^{k_n} \a^\top\t_{ni}\t_{ni}^\top\a = p\,\a^\top\[\frac{1}{k_n}\sum_{i=1}^{k_n}\t_{ni}\t_{ni}^\top\]\a.
\end{equation*}
For the sample covariance we obtain,
\begin{align}
&\frac{1}{k_n}\sum_{i=1}^{k_n}\t_{ni}\t_{ni}^\top = \frac{1}{k_n}\sum_{i=1}^{k_n}\frac{\T_n^{-1/2}\bm\Omega^{1/2}\w_i\w_i^\top\bm\Omega^{1/2}\T_n^{-1/2}}{\norm{\T_n^{-1/2}\bm\Omega^{1/2}\w_i}^2} \nonumber \\ 
&\quad = 
\T_n^{-1/2}\bm\Omega^{1/2}\[\frac{1}{k_n}\sum_{i=1}^{k_n} \frac{\w_i\w_i^\top}{\norm{\T_n^{-1/2}\bm\Omega^{1/2}\w_i}^2}\]\bm\Omega^{1/2}\T_n^{-1/2} \nonumber \\
&\qquad \xrightarrow{a.s.} \frac{1}{p}\I,\quad n \to \infty, 
\end{align}
and therefore,
\begin{equation}
\frac{1}{k_n}\sum_{i=1}^{k_n} X_{ni}^2 \;\xrightarrow{P}\; \a^\top\I\,\a = \norm{\a}^2 = 1,\quad n \to \infty.
\end{equation}
The rest of the proof is based on the argument proposed in \cite{soloveychik2020central}. Assume without loss of generality that $n$ is an even number and set
\begin{equation}
k_n = \frac{n}{2}.
\end{equation}
Consider now the following sequence,
\begin{equation}
Y_{ni} = \begin{cases} X_{ni}, & i\leqslant k_n, \\ -X_{ni}, & i> k_n. \end{cases}
\end{equation}
Clearly, the new sequence is exchangeable with the same joint distribution as the original sequence. Indeed, under $\mathcal{H}_0$ the joint distribution of $\{\t_{ni}\}$ is invariant under multiplication of any of the random vectors by $-1$ and $\T_n$ is an even function of $\t_{ni}$. Now all the conditions of Lemma \ref{eq:martin_theorem} are satisfied for the new sequence $\{Y_{ni}\}$ and we obtain,
\begin{equation}
\sqrt{\frac{n}{2}}\[\frac{2}{n}\sum_{i=1}^{n/2}Y_{ni}-\frac{1}{n}\sum_{i=1}^{n}Y_{ni}\] \xrightarrow{L} \mathcal{N}\(0,\frac{1}{2}\),\;\; n \to \infty.
\end{equation}
Note that
\begin{align}
&\frac{2}{n}\sum_{i=1}^{n/2}Y_{ni}-\frac{1}{n}\sum_{i=1}^{n}Y_{ni} \\
&\qquad= \frac{2}{n}\sum_{i=1}^{n/2}X_{ni}-\frac{1}{n}\[\sum_{i=1}^{n/2}X_{ni} - \sum_{i=n/2+1}^{n}X_{ni}\] \nonumber \\ 
&\qquad= \frac{1}{n}\sum_{i=1}^{n/2}X_{ni} + \frac{1}{n} \sum_{i=n/2+1}^{n}X_{ni} = \frac{1}{n}\sum_{i=1}^{n}X_{ni}, \nonumber
\end{align}
to obtain
\begin{equation}
\sqrt{\frac{n}{2}}\cdot\frac{1}{n}\sum_{i=1}^{n}X_{ni} \xrightarrow{L} \mathcal{N}\(0,\frac{1}{2}\),\;\; n \to \infty,
\end{equation}
or
\begin{equation}
\frac{1}{\sqrt{n}}\sum_{i=1}^{n}X_{ni} \xrightarrow{L} \mathcal{N}\(0,1\),\;\; n \to \infty.
\end{equation}
By the definition of $X_{ni}$ we get
\begin{equation}
\sqrt{p}\cdot\a^\top \frac{1}{\sqrt{n}}\sum_{i=1}^{n}\t_{ni} \xrightarrow{L} \mathcal{N}\(0,1\),\;\; n \to \infty.
\end{equation}
Finally, recall that the vector $\a$ was chosen arbitrarily to conclude the proof.
\end{proof}

For a sequence of vectors $\{\y_i\}_{i=1}^n$, denote their sample covariance by
\begin{equation}
\S_{\y,n} = \frac{1}{n}\sum_{i=1}^n\y_i\y_i^\top.
\end{equation}

\begin{lemma}[Theorem 6.2 from \cite{benedetto2003finite}]
\label{lem:frame_lemma}
Let $\{\y_i\}_{i=1}^n \subset \mathbb{S}^{p-1}$ be a set of $n\geqslant p$ vectors, then
\begin{equation}
\Tr{\S_{\y,n}^2} \geqslant \frac{1}{p}.
\end{equation}
\end{lemma}

\begin{proof}[Proof of Proposition \ref{prop:gine_exch}]
The difference in the behavior of the Ajne and Gin{\'e} statistics stems from the fact that the former is a sum of Gegenbauer polynomials of odd orders involving only monomials of odd powers, while the latter reads as a sum of Gegenbauer polynomials of even orders involving only monomials of even powers \cite{gine1975invariant, prentice1978invariant}. Next we explain this in more detail. 

Gegenbauer (ultraspherical) polynomial \cite{reimer2012multivariate} of index $\alpha$ and order $q \geqslant 2$ is defined as
\begin{equation}
\label{eq:gene_def_series}
C_{q}^\alpha(z) = \sum_{k=0}^{\floor{q/2}} (-1)^k\frac{\Gamma(q-k+\alpha)}{\Gamma(\alpha)k!(q-2k)!}(2z)^{q-2k}.
\end{equation}
Note that the Gegenbauer polynomials of odd/even order involves monomials of only odd/even order, respectively.

In order to analyze the Gin{\'e} statistic (\ref{eq:gine_stat_def}), we use the following expansion of $\sin\theta$ into Gegenbauer polynomials in $\cos\theta$ from \cite{prentice1978invariant},
\begin{align*}
&\frac{1}{2}-\frac{p-1}{2}\(\frac{\Gamma\(\alpha+\frac{1}{2}\)}{\Gamma(\alpha+1)}\)^2\sin\theta \\
&\qquad= \sum_{q=1}^\infty \frac{(p-1)(2q-1)(4q+p-2)}{(p-2)(2q+p-1)8\pi} \\
&\qquad\qquad\times\(\frac{\Gamma\(\alpha+\frac{1}{2}\)\Gamma\(q-\frac{1}{2}\)}{\Gamma(q+\alpha+\frac{1}{2})}\)^2C_{2q}^\alpha(\cos\theta),
\end{align*}
where we remind the reader that
\begin{equation}
\alpha = \frac{p}{2}-1.
\end{equation}
Using (\ref{eq:gene_def_series}) we can write,
\begin{align}
\label{eq:sin_decomp}
&\frac{1}{2}-\frac{p-1}{2}\(\frac{\Gamma\(\alpha+\frac{1}{2}\)}{\Gamma(\alpha+1)}\)^2\sin\theta = \sum_{r=0}^\infty \gamma_{2r}(\alpha, p)\cos^{2r}\theta \nonumber \\
&\qquad = \gamma_0(\alpha, p) + \gamma_2(\alpha, p)\cos^2\theta + \sum_{r=2}^\infty \gamma_{2r}(\alpha, p)\cos^{2r}\theta. \nonumber
\end{align}
Below we use the explicit form of $\gamma_2(\alpha, p)$,
\begin{align}
\gamma_2(\alpha, p) &= \sum_{q=1}^\infty \frac{(p-1)(2q-1)(4q+p-2)}{(p-2)(2q+p-1)8\pi} \nonumber \\
&\qquad\qquad\times\(\frac{\Gamma\(\alpha+\frac{1}{2}\)\Gamma\(q-\frac{1}{2}\)}{\Gamma\(q+\alpha+\frac{1}{2}\)}\)^2\zeta_{2q,2}^{\alpha},
\end{align}
where
\begin{equation}
\zeta_{2q,2}^{\alpha} = 2(-1)^{q-1}\frac{\Gamma(q+1+\alpha)}{\Gamma(\alpha)(q-1)!},
\end{equation}
is the weight of $z^2$ in $C_{2q}^\alpha(z)$ defined as in (\ref{eq:gene_def_series}). Thus, we obtain
\begin{align}
\gamma_2(\alpha, p) &= \sum_{q=1}^\infty (-1)^{q-1}\frac{(p-1)(2q-1)(4q+p-2)}{(p-2)(2q+p-1)4\pi} \\
&\qquad\times\(\frac{\Gamma\(\frac{p}{2}-\frac{1}{2}\)\Gamma\(q-\frac{1}{2}\)}{\Gamma\(q+\frac{p}{2}-\frac{1}{2}\)}\)^2\frac{\Gamma\(q+\frac{p}{2}\)}{\Gamma\(\frac{p}{2}-1\)(q-1)!}. \nonumber
\end{align}
Since the last series is telescopic we conclude that in particular
\begin{equation}
\gamma_2(\alpha, p) > 0.
\end{equation}
Let us compute the first few terms of this series,
\begin{align}
& \gamma_2(\alpha, p) \\
& \;\;= \frac{p(p+2)}{4(p+1)(p-1)} - \frac{3p(p+2)(p+6)}{8(p-1)(p+1)^2(p+3)} + O\(\frac{1}{p^2}\) \nonumber \\
& \;\;= \frac{1}{4} + \frac{1}{8p} + O\(\frac{1}{p^2}\). \nonumber
\end{align}
Recall that due to (\ref{eq:psi_n_not}),
\begin{equation}
\cos(\psi_{n,ij}) = \t_{ni}^\top\t_{nj},
\end{equation}
and therefore,
\begin{equation}
\sin(\psi_{n,ij}) = \sin(\psi_{n,ji}),
\end{equation}
together with
\begin{equation}
\sin(\psi_{n,ii}) = 0.
\end{equation}
Now we can see that the Gin{\'e} statistic reads as
\begin{align}
\label{eq:t_g_der_sum}
& t_{G}\(\{\t_{ni}\}\) = \frac{n}{2}-\frac{p-1}{2n}\(\frac{\Gamma\(\alpha+\frac{1}{2}\)}{\Gamma(\alpha+1)}\)^2\sum_{i<j}\sin(\psi_{n,ij}) \nonumber \\
& = \frac{n}{2} - \frac{p-1}{4n}\(\frac{\Gamma\(\alpha+\frac{1}{2}\)}{\Gamma(\alpha+1)}\)^2\sum_{i,j=1}^n\sin(\psi_{n,ij}) \nonumber \\
& = \frac{n}{2} - \frac{1}{2n}\[\frac{p-1}{2}\(\frac{\Gamma\(\alpha+\frac{1}{2}\)}{\Gamma(\alpha+1)}\)^2\sum_{i,j=1}^n\sin(\psi_{n,ij})\] \nonumber \\
& = \frac{n}{2} + \frac{1}{2n}\sum_{i,j=1}^n\Bigg[\gamma_0(\alpha, p) + \gamma_2(\alpha, p)\cos^2\theta \nonumber \\
& \qquad\qquad\qquad\qquad \left. + \sum_{r=2}^\infty \gamma_{2r}(\alpha, p)\cos^{2r}\theta - \frac{1}{2}\] \nonumber \\
& = \frac{n}{4} + \frac{n\gamma_0(\alpha, p)}{2} + \frac{1}{2n}\sum_{i,j=1}^n\Bigg[\gamma_2(\alpha, p)\(\t_{ni}^\top\t_{nj}\)^2 \nonumber \\
& \qquad\qquad\qquad\qquad \left.+ \sum_{r=2}^\infty \gamma_{2r}(\alpha, p)\(\t_{ni}^\top\t_{nj}\)^{2r}\].
\end{align}
Note that
\begin{align}
\sum_{i,j=1}^n\(\t_{ni}^\top\t_{nj}\)^2 &= \sum_{k,l=1}^p \[\sum_{i=1}^n \t_{ni}^{(k)} \t_{ni}^{(l)}\]^2 \\
&= \Tr{\[n \S_{\t,n}\]^2} = n^2\Tr{\S_{\t,n}^2}, \nonumber
\end{align}
where we denote
\begin{equation}
\t = \begin{bmatrix} \t^{(1)} \\ \vdots \\ \t^{(p)} \end{bmatrix}.
\end{equation}
In our setup, the sample covariance matrix satisfies the following relation,
\begin{align}
\S_{\t,n} &= \frac{1}{n}\sum_{i=1}^n\t_{ni}\t_{ni}^\top = \frac{1}{n}\sum_{i=1}^n\frac{\T_n^{-1/2}\x_i\x_i^\top\T_n^{-1/2}}{\norm{\T_n^{-1/2}\x_i}^2} \\ 
&= \frac{1}{n}\T_n^{-1/2}\sum_{i=1}^n\frac{\x_i\x_i^\top}{\x_i^\top\T_n^{-1}\x_i}\T_n^{-1/2} = \frac{1}{p}\I, \nonumber
\end{align}
and therefore,
\begin{equation}
\Tr{\S_{\t,n}^2} = \frac{1}{p^2}\Tr{\I} = \frac{1}{p}.
\end{equation}
By Lemma \ref{lem:frame_lemma}, for $\{\w_i\}_{i=1}^n$ i.i.d. uniformly distributed over $\mathbb{S}^{p-1}$ with $n\geqslant p$,
\begin{equation}
\Tr{\S_{\w,n}^2} \geqslant \frac{1}{p} = \Tr{\S_{\t,n}^2}.
\end{equation}
Since $\gamma_2(\alpha, p) > 0$, from (\ref{eq:t_g_der_sum}) we infer that $t_{G}\(\{\w_i\}\)$ first-order stochastically dominates $t_{G}\(\{\t_{ni}\}\)$.

Recall that the limiting distribution of the spectrum of $p\,\S_{\w,n}$ is given by the Marchenko-Pastur law \cite{marcenko1967distribution} whose second moment gives us the following asymptotic equivalence,
\begin{equation}
\mathbb{E}\[\Tr{\[p\,\S_{\w, n}\]^2}\] \sim p + \frac{p^2}{n}, \;\; n\to\infty.
\end{equation}
As a consequence,
\begin{multline*}
\frac{1}{2n}\(n^2\mathbb{E}\[\Tr{\S_{\w,n}^2}\] - n^2\mathbb{E}\[\Tr{\S_{\t,n}^2}\]\) \\
\sim \frac{n^2}{2n}\(\frac{1}{p} + \frac{1}{n} - \frac{1}{p}\) \sim \frac{1}{2},\;\; n\to\infty,
\end{multline*}
and from (\ref{eq:t_g_der_sum}) we conclude,
\begin{multline*}
\mathbb{E}\[t_{G}\(\{\w_i\}\)\] - \mathbb{E}\[t_{G}\(\{\t_{ni}\}\)\] \sim \frac{\gamma_2(\alpha, p)}{2} \\
\sim \frac{1}{8} + \frac{1}{16p} + O\(\frac{1}{p^2}\), \;\; n\to \infty. 
\end{multline*}
\end{proof}

\bibliographystyle{IEEEtran}
\bibliography{ilya_bib}

\end{document}

%% file: ilyas_commands.tex
\renewcommand{\(}{\left(}
\renewcommand{\)}{\right)}
\renewcommand{\[}{\left[}
\renewcommand{\]}{\right]}

\renewcommand{\a}{\mathbf{a}}

\renewcommand{\S}{\mathbf{S}}

\newcommand{\y}{\mathbf{y}}

\newcommand{\z}{\mathbf{z}}
\newcommand{\w}{\mathbf{w}}

\newcommand{\Z}{\mathbf{Z}}

\newcommand{\x}{\mathbf{x}}

\newcommand{\I}{\mathbf{I}}

\renewcommand{\t}{\mathbf{t}}

\newcommand{\T}{\mathbf{T}}

\newcommand{\X}{\mathbf{X}}
\newcommand{\Y}{\mathbf{Y}}

\newcommand{\Tr}[1]{{\rm{Tr}}\left(#1\right)}

\newcommand{\End}[1]{{\rm{End}}}

\renewcommand{\det}[1]{\left|#1\right|}

\newcommand{\var}[1]{{\rm{var}}\(#1\)}

\newtheorem{lemma}{Lemma}
\newtheorem{definition}{Definition}
\newtheorem{theorem}{Theorem}
\newtheorem{prop}{Proposition}
\newtheorem{corollary}{Corollary}

\newtheorem{rem}{Remark}

\newcommand{\norm}[1]{\left\lVert#1\right\rVert}